\documentclass[reqno]{amsart}
\usepackage{amsmath}
\usepackage{amssymb}
\usepackage{xypic}
\usepackage{mathrsfs}
\usepackage{hyperref}
\usepackage{graphicx}
\usepackage{perpage}
\MakePerPage{footnote}
\usepackage{setspace}

\newtheorem{theorem}{Theorem}[section]
\newtheorem{lemma}[theorem]{Lemma}

\theoremstyle{remark}

\title[Existence of definable Whitney stratifications]{A geometric proof of the existence of definable Whitney stratifications}

\author{Nhan Nguyen, Saurabh Trivedi and David Trotman}

\address{Nhan Nguyen $\&$ David Trotman, LATP (UMR 7353), Centre de Math\'ematiques et Informatique, Aix-Marseille Universit\'e, 39 rue Joliot-Curie, 13453 Marseille Cedex 13, France.}

\address{Saurabh Trivedi, Instytut Matemaczny PAN, Division in Krakow, ul. Sw. Tomasza, 30, 31-027 Krakow, Poland.}

\newcommand{\bb}{\mathbb}
\newcommand{\al}{\mathcal}

\begin{document}
\maketitle

\parskip .12cm

\begin{abstract} We give a geometric proof of existence of Whitney stratifications of definable sets in o-minimal structures. 
\end{abstract}

\section{Introduction}

It has been known for a long time that semi-varieties (semi-analytic or semi-algebraic for example) can be stratified into smooth manifolds satisfying Whitney conditions $(a)$ and $(b)$. Methods of doing this can be found in Whitney \cite{Whitney}, Wall \cite{Wall3}, Bochnak, Coste and Roy \cite{BCR}, {\L}ojasiewicz \cite{Lojasiewicz},  {\L}ojasiewicz, Stasica and Wachta \cite{LSW}, etc. All of the proofs given in the above mentioned literature of the existence of such stratifications use analytical techniques.

Kaloshin \cite{Kaloshin2} has claimed a geometric proof of the existence of stratifications of semivarieties satisfying the Whitney conditions. We show by giving a very simple counterexample that there is a gap in this proof of Kaloshin. In this article, motivated by the idea of Kaloshin, we give a geometric proof of the existence of these stratifications in the more general o-minimal setting. Our method  fills the gap in Kaloshin's proof and moreover it works for the case of definable sets in o-minimal structures. Loi \cite{Loi1} also proved this result with a different proof using a wing lemma.

Let us first describe the overview of the idea of Kaloshin:

The following terminology is due to Kaloshin. Let $V \subset \mathbb R^n$ be a closed semivariety and let $\Sigma$ be a stratification of $V$. Given strata $X$ and $Y$ of $\Sigma$ and a point $y \in \overline{X}\cap Y$, by a local connected component of $X$ at $y$ is meant a connected subset of $X$ obtained from intersecting $X$ by a sufficiently small open ball centered at $y$. By a result of {\L}ojasiewicz \cite{Lojasiewicz}, there exist finitely many such connected components for any point $y \in Y$.

A local connected component $X_\alpha$ is said to be an essential component of $X$ at $y$ if $y$ lies in the interior of $Y \cap \overline{X_{\alpha}}$ (considered as a subset of $Y$). Now $Sing_a(X,Y)$ is defined as the set of points $y\in Y$ such that the union of the essential components of $X$ at $y$ is not $(a)$-regular over $Y$ at $y$. Kaloshin proves that the set $Sing_a(X,Y)$ is a semivariety and has dimension less than the dimension of $Y$, so showing that Whitney's condition $(a)$ is generic, and the result follows.

We will show pictorially that the set of $(a)$-faults (points where the condition $a$ fails) of a pair of strata $(X,Y)$ is in general bigger than $Sing_a(X,Y)$, and that considering only the essential components leaves several $(a)$-faults unaccounted for.

Consider the closed subset $V$ of $\bb R^3$ as in Figure \ref{counter}. It is like Santa's hat except that the conical tip is attached to the round edge of the hat.

\begin{figure}[hi]
\begin{center}
\includegraphics[scale=.85]{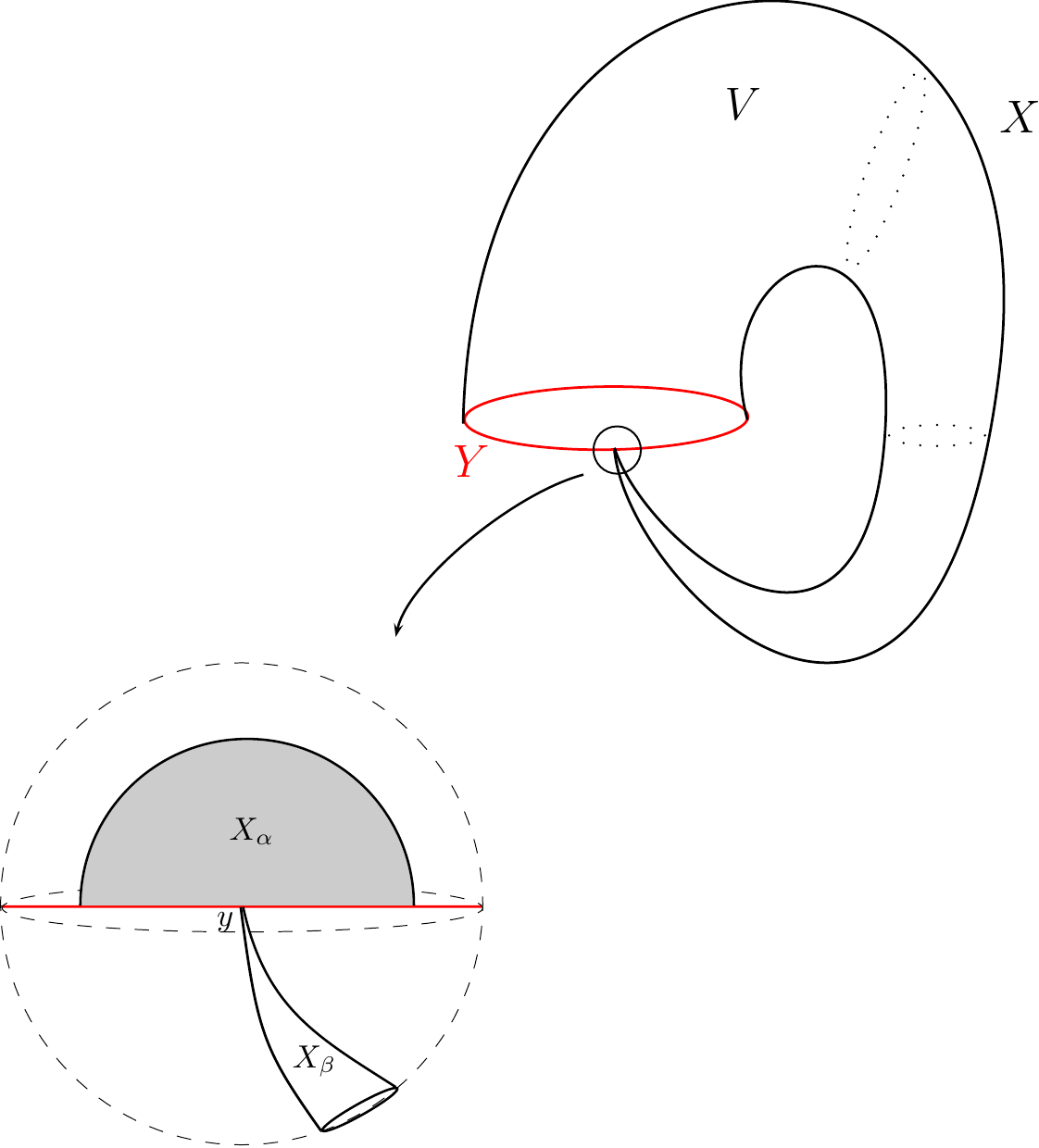}
\caption{}\label{counter}
\end{center}
\end{figure}

Applying the procedure of stratifying $V$ due to Wall \cite{Wall3}\footnote{We must mention here that Wall's method works only for closed semi-varieties.}  we find that $\bb R^ 3$ will have three strata compatible with $V$. The three dimensional stratum will be the complement of $V$ in $\bb R^ 3$. The two dimensional stratum will be $X$ and the one dimensional stratum will be $Y$.

Now, take $y \in Y$ as in the Figure \ref{counter} (the tip of the hat) and intersect $V$ with a small ball around $y$. We find that $X$ has two local connected components at $y$, denoted $X_{\alpha}$ and $X_{\beta}$. Notice that $X_{\alpha}$ is an essential component of $X$ near $Y$ while $X_{\beta}$ is not. Thus the set $Sing_a(X,Y)$ is empty. Notice also that $X$ is not $(a)$-regular over $Y$ at $y$. Thus the set of $(a)$-faults in this stratification of $V$ is strictly bigger than the set $Sing_a(X,Y)$.

We will now summarize the contents of the article. 

In section \ref{sec2} we give definitions of o-minimal structures, definable stratifications, stratifying conditions, Whitney conditions and state the main result (Theorem \ref{thm23}). The idea of the proof is to show that Whitney conditions are stratifying conditions (Lemma \ref{lem25} and \ref{lem26}).

In section \ref{sec3} we define Kuo functions. These functions give criteria to test Whitney conditions $(a)$ and $(b)$ in a stratification.

In section \ref{sec4} we prove that the Whitney conditions $(a)$ and $(b)$ are stratifying conditions. The key to the proof is the existence of a sequence of points in a stratum converging to a point in another stratum in its boundary such that the limit of the sequence of values of the Kuo functions on these points vanish (Lemma \ref{existence_sequence_lem} and \ref{lem_pb}).

\section{Preliminaries and statement of results}\label{sec2}

\subsection{o-minimal structures} A \emph{structure} on the ordered field $(\bb R,+,.)$ is a family $\al D = (D_n)_{n\in \bb N}$ satisfying the following properties:

1. $D_n$ is a boolean algebra of subsets of $\bb R^n$,

2. If $A \in D_n$ then $\bb R \times A \in D_{n+1}$ and $A \times \bb R \in D_{n+1}$,

3. $D_n$ contains the zero sets of all polynomials in $n$ variables,

4. If $A \in D_n$ then its projection onto the first $n-1$ coordinates in $\bb R^{n-1}$ is in $D_{n-1}$.

Such a $\al D$ is said to be \emph{o-minimal} if in addition

5. Any set $A \in D_1$ is a finite union of open intervals and points.

Elements of $D_n$ for any $n$ are called \emph{definable sets} of $\al D$. A map between two definable sets is said to be a \emph{definable map} if its graph is a definable set. 

Let $\al D$ be an o-minimal structure on $\bb R$. In what follows by definable we mean in this $\al D$.

\subsection{Definable stratifications and stratifying conditions}
A definable $C^p$-stratification $\Sigma$ of $\bb R^n$ is a partition of $\bb R^n$ into finitely many definable $C^p$ submanifolds \footnote{A definable $C^p$ submanifold of $\mathbb{R}^n$ meaning a definable subset and also a $C^p$ submanifold of $\mathbb{R}^n$.} of $\bb R^n$, called strata, such that the boundary of every stratum is either empty or a union of some other strata.

Let $\al A = \{A_1, \ldots, A_k\}$ be a family of definable subsets of $\bb R^n$. A stratification $\Sigma$ of $\bb R^n$ is said to be compatible with $\al A$ if each $A_i$ is the union of some strata of $\Sigma$. In the rest of the paper, by definable we mean of class $C^p$.

Let $(X,Y)$ be a pair of definable submanifolds of $\bb R^n$ such that $Y \subset \overline{X} \setminus X$. Let $\gamma$ be a condition on the pair $(X,Y)$ at points in $Y$. A point $y \in Y$ is said to be a $(\gamma)$-fault if the condition $\gamma$ fails to be satisfied for the pair $(X,Y)$ at $y$. We denote by $\al F_{\gamma}(X,Y)$ the set of all $(\gamma)$-faults for the pair $(X,Y)$. If $F_{\gamma}(X,Y)$ is empty then we say that the pair $(X,Y)$ is $(\gamma)$-regular. Moreover, a stratification is said to be $(\gamma)$-regular if every pair of its strata is $(\gamma)$-regular.

A condition $(\gamma)$ is said to be a \emph{stratifying condition} if for any pair $(X,Y)$ as above the set $\al F_{\gamma}(X,Y)$ is definable and $\dim \al F_{\gamma}(X,Y)< \dim Y$. Using cell decomposition theorem \cite{Dries1} and arguments as in the proof of Proposition 2 in \cite{LSW}, we have the following result (see also \cite{Loi3}).

\begin{theorem}\label{thm22} Let $\al A = \{A_1,\ldots, A_k\}$ be a family of definable subsets of $\bb R^n$. If $(\gamma)$ is a stratifying condition then there exists a $(\gamma)$-regular definable stratification of $\bb R^n$ compatible with $\al A$.  
\end{theorem}

\subsection{Whitney conditions} Let $X$ be a definable submanifold of $\mathbb{R}^n$ and $y \in \overline{X}$. A sequence of points $\{x_n\}$ in $X$ converging to $y$ is said to be a \emph{good sequence} if the corresponding sequences $\{T_{x_n}X\}$ of tangent spaces in the Grassmannian converges. The limit $\lim_{n \to \infty} T_{x_n}X$ will be called the \emph{Grassmannian limit} of the sequence $\{x_n\}$. Since the Grassmannian is a compact metric space, for every sequence in $X$ there exists a subsequence which is a good sequence. 

Let $(X,Y)$ be a pair of definable submanifolds of $\mathbb{R}^n$ such that $Y \subset \overline{X}\setminus X$. Consider the following conditions on $(X,Y)$ at a point $y \in Y$.

\begin{enumerate}
\item[(a)] the Grassmannian limit of every good sequence $\{x_n\}$ in $X$ converging to $y$ contains the tangent space $T_y Y$.

\item[(b)] for every sequence $\{y_n\}$ in $Y$ converging to $y$, the Grassmannian limit of every good sequence $\{x_n\}$ in $X$ converging to $y$ contains $v: = \lim_{n \to \infty} \frac{x_n - y_n}{\|x_n - y_n\|}$ if $v$ exists. 
\end{enumerate}

The reader must have realized that the conditions $(a)$ and $(b)$ are the usual Whitney conditions $(a)$ and $(b)$ written differently.

\begin{theorem}\label{thm23} Let $\mathcal{A} = \{A_1, \ldots A_k\}$ be a family of definable subsets of $\mathbb{R}^n$. Then there exists an $(a)$-regular (resp. $(b)$-regular) definable stratification of $\mathbb{R}^n$ compatible with $\mathcal{A}$.
\end{theorem}

By Theorem \ref{thm22}, to prove Theorem \ref{thm23}, it suffices to show that conditions $(a)$ and $(b)$ are stratifying conditions. For any definable submanifolds $X,Y \subset \bb R^n$ such that $Y \subset \overline{X} \setminus X$, it is easy to see that the set of $(a)-$faults $\mathcal{F}_a(X, Y)$ (resp.  $(b)$-faults $\mathcal{F}_b(X,Y)$) is definable once we write it using quantifiers, see for example \cite{Loi3}. Thus, we need to prove the following lemmas:

\begin{lemma}\label{lem25} For any definable submanifolds $X,Y \subset \bb R^n$ such that $Y \subset \overline{X} \setminus X$, we have $\dim \mathcal{F}_a(X, Y) < \dim Y$.
\end{lemma}

\begin{lemma}\label{lem26} For any definable submanifolds $X,Y \subset \bb R^n$ such that $Y \subset \overline{X} \setminus X$, we have $\dim \mathcal{F}_b(X, Y) < \dim Y$.
\end{lemma}

We will prove Lemma \ref{lem25} and \ref{lem26} in section \ref{sec4}.

\section{Kuo functions}\label{sec3}

Let $X, Y$ be definable submanifolds of $\bb R^n$ such that $Y \subset \overline{X} \setminus X$. Suppose that $\dim Y = k$. Since $(a)$ (resp. $(b)$) regularity is a local property we can assume that locally $Y$ is a $k$-plane with a basis of unit vectors $\{e_1, \ldots, e_k\}$.

Given a linear subspace $L$ of $\mathbb{R}^n$ we denote by $\pi_L : \mathbb{R}^n \rightarrow L$ the canonical orthogonal projection of $\mathbb{R}^n$ onto $L$. Let $x \in X$ and consider $T_xX$ as a linear subspace of $\bb R^n$. Using the idea of Kuo \cite{Kuo2} (see also \cite{Kaloshin2}) we define functions, which we call Kuo functions, that give criteria to test $(a)$ and $(b)$-regularity.

Let $p_{a} : X \rightarrow \mathbb{R}$ be the function defined by
$$p_{a}(x): = \sum_{i=1}^k\|\pi_{N_xX}(e_i)\|^2$$
where $N_xX$ is the orthogonal complement of $T_xX$.

Let $p_{b'} : X \rightarrow \mathbb{R}$ be the function defined by
$$p_{b'}(x): = \|\pi_{N_xX}(p(x))\|^2,$$ where $\displaystyle p(x): = \frac{x - \pi_{Y}(x)}{\|x-\pi_{
Y}(x)\|}$.

Let $p_b:X \rightarrow \mathbb{R}$ be the function defined by
$$p_b(x): = p_{a}(x) + p_{b'} (x).$$

Kuo \cite{Kuo2} (see also \cite{Kaloshin2}) proved that a pair $(X, Y)$ satisfies the condition $(a)$ (resp. $(b)$) at $y \in Y$ if and only if for every good sequence $\{x_n\}$ in $X$ converging to $y$ , $\lim_{n\to \infty}p_a(x_n) = 0$ (resp. $\lim_{n\to \infty}p_b(x_n) = 0$).  

\section{Existence of Whitney stratifications for definable sets\label{sec4}}

Let $P$ and $Q$ be linear subspaces of $\mathbb{R}^n$. The \emph{angle} between $P$ and $Q$ is defined by
$$\delta(P, Q): = \sup_{\lambda \in P, \| \lambda \| = 1} \{ \| \lambda - \pi_Q(\lambda) \|\}.$$

The function $\delta$ takes values in $[0,1]$. In general $\delta$ is not symmetric, for instance, if $P \subset Q$ then $\delta(P, Q) = 0$ while $\delta(Q, P) = 1$. The following properties are easy to verify. 

1. If $\dim P  = \dim Q$ then $\delta(P,Q) = \delta(Q, P)$. 

2. If $P \subset Q$ then $d(P, Q) = 0$. 

3. If $\dim T \leq \dim P \leq \dim Q$ then $\delta(T, Q) \leq \delta(T, P) + \delta(P, Q)$.

For a real number $\epsilon >0$, a definable submanifold $X$ is said to be $\epsilon$-flat if for every $x, x'$ in $X$,  $\delta (T_xX, T_{x'}X) < \epsilon$. If $\dim X = 0$ then we assume that $X$ is $\epsilon$-flat for every $\epsilon >0$.

\begin{lemma}\label{lemma_flat}
 Let $X \subset \mathbb{R}^n$ be a definable set of dimension $k < n$ and let $\epsilon>0$ be a real number. There is a definable stratification of $X$ such that every stratum is $\epsilon$-flat.
\end{lemma}
\begin{proof} This is proved for subanalytic sets in Proposition 5 in Kurdyka \cite{Kurdyka1}, but the idea also works for definable sets.  
\end{proof}

\begin{lemma}\label{existence_sequence_lem}
Let $X, Y$ be definable submanifolds of $\mathbb{R}^n$ such that $Y \subset \overline{X}\setminus X$ and let $y$ be a point in $Y$. Then there exists a good sequence $\{x_n\}$ in $X$ converging to $y$ such that $p_a(x_n)$ converges to $0$.
\end{lemma}

\begin{proof} Suppose on the contrary that there is an $\epsilon > 0$ such that for every good sequence $\{x_n\}$ in $X$ converging to $y$, the limit of the sequence $p_a(x_n)$ is greater than $\epsilon$. In other words, we can choose $\epsilon$ sufficient small such that for any given good sequence $\{x_n\}$ with the Grassmannian limit $\tau$, we have $\delta(T_yY, \tau) > \epsilon$.

Take a stratification of $\mathbb{R}^n$ compatible with $X$ such that its strata are $\frac{\epsilon}{4}$-flat (this is possible by Lemma \ref{lemma_flat}). We can write $X = \bigsqcup_{i=1}^m X_i$ where the $X_i$'s are the strata. Set $Y'  := \bigcup_{i=1}^m Int_Y (\overline{X}_i \cap Y)$. Notice that $Y'$ is open and dense in $Y$. The proof now breaks into the two following cases.

\textit{\underline{Case 1:}} $y \in Y'$.

There is an $X_i$, $1\leq i \leq m$, such that $y \in Int_Y (\overline{X}_i \cap Y)$. Fix a good sequence $\{x_n\}$ in $X_i$ and denote by $\tau$ its Grassmannian limit. 

Since $\delta(T_yY, \tau) > \epsilon$, we can choose a line $l \subset T_yY$ satisfying $\delta( l, \tau) > \frac{\epsilon}{2}$. We define the $\frac{\epsilon}{4}$-cone around $l$ centered at $y$ as follows:  
$$C_y : = \{x\in \mathbb{R}^n: \delta (\mu( x - y),  l ) < \frac{\epsilon}{4}\},$$ 
where $\mu(x-y)$ denotes the line spanned by the unit vector  $\frac{x-y}{\|x-y\|}$.

Since  $y \in Int_Y(\overline{X_i} \cap Y)$, the intersection $X_i(y): = X_i \cap C_y$ is a non-empty  definable set and $y \in \overline{X_i(y)}$. The curve selection lemma (see van den Dries \cite{Dries1}) says that there is a $C^1$ curve $\gamma: (0,1) \to X_i(y)$ such that $\lim_{t\to 0} \gamma(t) = y$. Choose a good sequence $\{x'_n\}$ along the curve $\gamma$ converging to $y$ and denote by $\tau'$ its Grassmannian limit. Put $l': = \lim_{n \to \infty} T_{x'_n}\gamma$, then $l' \subset \tau'$ and
$$\delta(l, \tau') \leq \delta(l, l') \leq \frac{\epsilon}{4} .$$

Since $X_i$ is $\frac{\epsilon}{4}$-flat, $\delta(\tau, \tau') < \frac{\epsilon}{4}$. Thus,
$$\delta(l, \tau) \leq \delta(l, \tau') + \delta(\tau', \tau) < \frac{\epsilon}{4} + \frac{\epsilon}{4} = \frac{\epsilon}{2},$$
a contradiction.

\textit{\underline{Case 2:}} $y \not\in Y'$.

Because $Y'$ is dense in $Y$ we can find a sequence $\{y_n\}$ in $Y'$ tending to $y$. By case 1, for each $y_n$ there is a good sequence $\{x_{n,m}\}$ in $X$ converging to $y_n$ such that $p_a(x_{n,m})$ converges to $0$. It is possible to choose a good sequence $\{x'_n\}$ in $X$ converging to $y$ such that  $x'_n \in \{x_{n,m}\}$ and $p_a(x'_n) < \epsilon$. The limit of the sequence $p_a(x'_n)$  is clearly less than  $\epsilon$. This provides a contradiction.

\end{proof}

To prove Lemma \ref{lem25} we need the following definitions. For $y \in Y$, denote by $B_r(y)$ the open ball in $\mathbb{R}^n$ of radius $r$ centered at $y$. By Hardt's theorem about topological triviality for definable sets (Theorem 5.19, page 60 in \cite{Coste1}), the topological type of the intersection $B_r(y)\cap X$ is stable, i.e. there is an $r> 0$ sufficiently small such that for every $0< r'< r$ the sets $B_r(y)\cap X$ and $B_{r'}(y)\cap X$ are topologically equivalent. Denote by $N_y$ the number of connected components of the intersection $B_r(y)\cap X$. This number is uniformly bounded on $Y$. More precisely, there exists an integer $\kappa$ such that $N_y \leq \kappa$ for all $y \in Y$. A connected component $X_i(y)$ ($i=1,\ldots,N_y$) of the intersection $B_r(y)\cap X$ is said to be \emph{essential} if $y$ is in the interior of $\overline{X_i(y)}\cap Y$ in $Y$, denoted by $Int_Y(\overline{X_i(y)}\cap Y)$, $(i = 1,\ldots, N_y)$. We say that $y$ is an \emph{essential point} if $X_i(y)$ is essential for all $i$. 

%===================================
 Observe that every point in $\displaystyle \bigcap_i^{N_y} Int_Y (\overline{X_i(y)} \cap Y)$ is essential. Set $T_j (X,Y): = \{y\in Y: N_y = j\}$. Then the set of essential points can be written as follows
$$\Omega(X,Y): = \bigcup_{j= 1} ^ \kappa \{ y \in T_j(X,Y):  y \in \bigcap_{i \leq j} Int_Y (\overline{X_i(y)} \cap Y) \}.$$ This implies that $\Omega(X,Y)$ is an open definable set in $Y$.  In addition, we can cover $Y$ by countably many balls $B_{r_\alpha}(y_\alpha)$ where $y_\alpha \in Y \cap ( \{0\}^{n-k} \times \mathbb{Q}^k)$ and $r_\alpha \in \mathbb{Q}$ such that the intersection $B_{r_\alpha}(y_\alpha) \cap X$ is stable. It is clear that the set of non-essential points has dimension less than the dimension of $Y$ since it is contained in the countable union of boundaries of $\overline{X_i(y_\alpha)}\cap Y$ in $Y$  for all $y_\alpha$ and all $i = 1,\ldots, N_{y_\alpha}$. The set $\Omega(X,Y)$ thus is a definable set open and dense in $Y$. 

%The set of non-essential points has to be in the union of boundaries of $\overline{X_i(y_\alpha)}\cap Y$ in $Y$  for all $y_\alpha$ and all $i = 1,\ldots, N_{y_\alpha}$. It is the union of countably many sets of dimensions less than the dimension of $Y$, so its dimension is less than the dimension of $Y$. The set $\Omega(X,Y)$ thus is a definable set open and dense in $Y$. 
%====================================

\begin{proof}[Proof of Lemma \ref{lem25}] Since the set of essential points in $Y$ is definable, dense and open in $Y$, we can assume without loss of generality that every point in $Y$ is essential.

 Take a point $y$ in $\mathcal{F}_a(X,Y)$. By Lemma \ref{existence_sequence_lem}, there is an essential component $X_i(y)$ with two sequences of points $\{x'_n\}$ and $\{x''_n\}$ converging to $y$ such that  $p_a(x_n) \to \epsilon'$ and $p_a(x'_n) \to \epsilon''$ for some non-negative numbers $\epsilon' < \epsilon''$. Notice that the function $p_a(x)$ takes values in $[0, k]$ where $k$ is the dimension of $Y$. By Sard's lemma there exists a regular value $\epsilon \in (\epsilon', \epsilon '')$ of the function $p_a$, so the set $X^{\epsilon}:= (p_a)^{-1} (\epsilon)$ is a definable submanifold of $X$ of codimension $1$ in $X$. Since $X_i(y)$ is locally connected at $y$, $x_n'$ and $x_n ''$ can be connected by a curve $\gamma_n$. Choosing points $x_n \in \gamma_n$ such that $p_a(x_n) = \epsilon$, we get a sequence $\{x_n\} \subset X^\epsilon$ converging to $y$, and hence $y \in \overline{X^\epsilon}\setminus X^\epsilon$.

Now choose countably many regular values $\{\epsilon_\nu\}_{\nu \in \mathbb{Z}}$ of the function $p_a$ whose union is dense in $[0,k]$ and define $X^{\epsilon_\nu} :=( p_a)^{-1}(\epsilon_\nu)$. Then the union $\bigcup_{\nu \in \mathbb{Z}} \overline{X^{\epsilon_\nu}}\setminus X^{\epsilon_\nu}$ contains all $(a)$-faults of the pair $(X,Y)$.

Put $I =: \{ \nu \in \mathbb{Z}:\dim \overline{X^{\epsilon_\nu}} \cap Y = \dim Y\}$. For $\nu \in I$, denote by $Y^{\epsilon_\nu} = Int_Y(\overline{X^{\epsilon_\nu}} \cap Y)$, then $(X^{\epsilon_\nu}, Y^{\epsilon_\nu})$ is again a pair of definable submanifolds with $Y^{\epsilon_\nu} \subset \overline{X^{\epsilon_\nu}}\setminus X^{\epsilon_\nu}$. 
Let $p_a^{\epsilon_\nu}$ be the Kuo function on $(X^{\epsilon_\nu}, Y^{\epsilon_\nu})$ constructed as in section 3. Observe that $p_{a}^{\epsilon_\nu}(x) \geq p_a(x)$ for every $x \in X^{\epsilon_\nu}$. This shows 
$$\mathcal{F}_a(X,Y) \subset \bigcup_{\nu\in I} \mathcal{F}_a(X^{\epsilon_\nu}, Y^{\epsilon_\nu}) \cup Z,$$
where $Z: = \cup_{\nu \in \mathbb{Z}\setminus I} (\overline{X^{\epsilon_\nu}} \cap Y) \bigcup \cup_{\nu \in I} \bigg((\overline{X^{\epsilon_\nu}}\cap Y)\setminus Y^{\epsilon_\nu}\bigg)$ a subset of codimension greater than $1$ in $Y$.

Because a countable union of subsets of positive codimensions in $Y$ is a subset of positive codimension in $Y$, it remains to show that $\dim \mathcal{F}_a (X^{\epsilon_\nu}, Y^{\epsilon_\nu}) < \dim Y^{\epsilon_\nu}$ for $\nu \in I$. This follows from the inductive application of the above arguments for $(X^{\epsilon_\nu}, Y^{\epsilon_\nu})$. The induction stops when $\dim X^{\epsilon_\nu} \leq \dim Y$. 
\end{proof}

On order to prove Lemma 2.4 we will use the following Lemma 4.3 which plays the same role for $b$ as Lemma 4.2 does for $(a)$.

\begin{lemma}\label{lem_pb}
Let $X, Y$ be definable submanifolds of $\mathbb{R}^n$ such that $Y \subset \overline{X}\setminus X$ and let $y$ be a point in $Y$. There is a good sequence $\{x_n\}$ in $X$ converging to $y$ such that $p_{b}(x_n)$ converges to $0$.
\end{lemma}

\begin{proof}
Suppose that there is an $\epsilon > 0$ such that $p_{b}(x_n)> \epsilon$ for every good sequence $\{x_n\} $ in $X$ converging to $y$. We will show a contradiction by giving a sequence $\{x_n\}$ in $X$ converging to $y$ such that $p_{b}(x_n)$ converges to a value less than $\epsilon$.

For $y \in Y$, we define $Z(y) : = X \cap (Y^\perp + \{y\})$ and
$$\omega(y): = \inf\{d(x, y): x \in Z(y)\},$$ where $d(x,y)$ is the usual distance from $x$ to $y$. Put $\omega(y)=  1$ if $Z(y)  = \emptyset$. Clearly $\omega(y)$ is a definable function on $Y$. 

We claim that the set $\Delta: = \{y \in Y: y \in \overline{Z(y)}\}$ is open and dense in $Y$. In other words, its complement $\Delta^c: = \{y\in Y: \omega(y) > 0\}$ is of dimension less than the dimension of $Y$. Thus, suppose to the contrary that $\dim \Delta^c = \dim Y$. By the cell decomposition theorem and local compactness of $Y$, there is an open set $U$ in $Y$ and a constant $c > 0$ such that $\omega(y) > c$ for every $y \in U$. This means $U \not\subset \overline{X}\setminus X$, a contradiction.

Denote by $Y'$ the set of points in $\Delta$ which are not ($a$)-faults. Take $y \in Y'$. The curve selection lemma says that there exists a $C^1$ definable curve $\gamma: (0,1) \to Z(y)$ such that $\lim_{t \to 0} \gamma(t) = y$. Choose $\{z_m\}$ a good sequence in the curve $\gamma$ converging to $y$ and denote by $\tau$ its Grassmannian limit. From the construction we have  $\pi_Y(z_m) = y$, hence $p(z_m)= \dfrac{z_m - y}{\|z_m - y\|}$. Obviously $\lim_{m \to \infty} p(z_m) \in \lim_{m \to \infty} T_{z_m} \gamma \subset \tau$. This implies that $p_{b'}(z_m) = 0$. Moreover, since $y'$ is not an $(a)$-fault, $p_b(z_m) = 0$.

Since $Y'$ is dense in $Y$, for $y \in Y$ there is a sequence $\{y_n\} \subset Y'$ converging to $y$. Let $\{\gamma_n\}$ be the corresponding sequence of curves as above. Choose a sequence ${x_n}$ converging to $y$ with $x_n \in \gamma_n$ and $p_{b}(x_n) < \epsilon$, then $\lim_{n\to \infty} p_b(x_n)$ is obviously less than $\epsilon$. This gives a contradiction.
\end{proof}

Lemma \ref{lem_pb} together with the arguments of the proof of Lemma \ref{lem25} provide a proof for Lemma \ref{lem26}. 

\bibliographystyle{plain}
\bibliography{mainbibliography}

\end{document}